\documentclass[12pt]{amsproc}
\usepackage[utf8]{inputenc}
\usepackage[english]{babel}
\usepackage{amsmath}
\usepackage{amsfonts}
\usepackage{amssymb}
\usepackage{mathrsfs}
\usepackage[pdftex]{color,graphicx}

\newtheorem{theorem}{\bf Theorem}[section]
\newtheorem{lemma}[theorem]{\bf Lemma}

\newcommand{\E}{\mathscr{E}} 

\author[C. Acciarri]{Cristina Acciarri}
\address{Department of Mathematics, University of Brasilia, 70910-900 Bras\'ilia DF, Brazil}
\email{acciarricristina@yahoo.it}

\author[P. Shumyatsky]{Pavel Shumyatsky}
\address{Department of Mathematics, University of Brasilia, 70910-900 Bras\'ilia DF, Brazil}
\email{pavel@unb.br}

\author[D. Silveira ]{Danilo  Silveira}
\address{Department of Mathematics, Federal University of Goias, 75704-020 Catal\~ao GO, Brazil}
\email{sancaodanilo@ufg.br}


\thanks{This work was supported by the Conselho Nacional de Desenvolvimento Cient\'{\i}fico e Tecnol\'ogico (CNPq),  and Funda\c c\~ao de Apoio \`a Pesquisa do Distrito Federal (FAPDF), Brazil.}

\title[Engel sinks of fixed points  in finite groups ]{Engel sinks of fixed points \\ in finite groups }

\begin{document}

\begin{abstract} For an element $g$ of a  group $G$, an Engel sink is a subset  $\E(g)$  such that for every $ x\in G $ all sufficiently long commutators $ [x,g,g,\ldots,g] $ belong to $\E(g)$. Let $q$ be a prime, let $m$ be  a positive integer and $A$ an elementary abelian group of order $q^2$ acting coprimely on a finite group $G$. We show that if  for each nontrivial element  $a$  in $ A$  and every element $g\in C_{G}(a)$ the cardinality of the smallest Engel sink $\E(g)$ is at most $m$, then 
the order of $\gamma_\infty(G)$ is bounded in terms of $m$ only.  Moreover we  prove that  if  for each $a\in A\setminus \{1\}$ and  every element $g\in C_{G}(a)$, the smallest Engel sink $\E(g)$ generates a subgroup of rank at most $m$, then 
the rank of $\gamma_\infty(G)$ is bounded in terms of $m$ and $q$ only.

\end{abstract}

\maketitle

\section{Introduction}

A group $G$ is called an \emph{Engel group} if for every $x,g\in G$ the equation $[x,g,g,\dots , g]=1$ holds, where $g$ is repeated in the commutator sufficiently many times depending on $x$ and $g$. (Throughout the paper, we use the left-normed simple commutator notation
$[a_1,a_2,a_3,\dots ,a_r]=[...[[a_1,a_2],a_3],\dots ,a_r]$.)
The classical theorem of Zorn states that a finite Engel group is nilpotent \cite{Zorn}.  In a number of recent papers groups that are `almost Engel' in the sense of restrictions on so-called Engel sinks were considered \cite{KS-2016,KS-2018,KS-rank}. An \emph{Engel sink} of an element $g\in G$ is a set $\E(g)$ such that for every $x\in G$ all sufficiently long commutators $[x,g,g,\dots ,g]$ belong to $\E(g)$, that is, for every $x\in G$ there is a positive integer $n(x,g)$ such that
 $$[x,\underbrace{g,g,\dots ,g}_n]\in \E(g)\qquad \text{for all }n\geq n(x,g).
 $$
Engel groups are precisely the groups for which we can choose $\E(g)=\{ 1\}$ for all $g\in G$. In \cite{KS-2018} finite, profinite, and compact groups in which every element has a finite Engel sink were considered. It was shown that compact groups with this property are finite-by-(locally nilpotent). Results for finite groups have to be of quantitative nature. Obviously, in a finite group every element has the smallest Engel sink, so from now on we use the term Engel sink for the minimal Engel sink of $g$, denoted by $\E(g)$, thus eliminating ambiguity in this notation. One of the results obtained in \cite{KS-2018} says that if $G$ is a finite group and there is a positive integer $m$ such that $|\E(g)|\leq m$ for all $g\in G$, then $G$ has a normal subgroup of order bounded in terms of $m$ with nilpotent quotient. A subsequent paper \cite{KS-rank} dealt with finite groups in which there is a bound on the ranks of the subgroups generated by the Engel sinks. Here, the \textit{rank} of a finite group is the minimum number $r$ such that every subgroup can be generated by $r$ elements. Let $\gamma_\infty(G)=\cap_{i=1}^{\infty} \gamma_i(G)$ be the intersection of all terms of the lower central series of a group $G$ (the nilpotent residual of $G$).
 The following  results were obtained  in \cite[Theorem 3.1]{KS-2018} and  \cite[Theorem 1.1]{KS-rank}, respectively.

\begin{theorem}\label{KStheorems}
Let $ G $ be a finite group and $m$ a positive integer.
\begin{itemize}
\item[(1)] If $\E(g)$ has at most $m$ elements for every $g\in G$, then the order of $\gamma_\infty(G)$ is bounded in terms of $m$ only.
\item[(2)] If $\langle\E(g)\rangle$ has rank at most $m$ for every $g\in G$, then the rank of $\gamma_\infty(G)$ is bounded in terms of $m$ only.
	\end{itemize}
\end{theorem}
As usual, $\langle X\rangle$ denotes the subgroup generated by a subset $X$ of $G$. In the present article we consider finite groups $G$ admitting a coprime action by an elementary abelian $q$-group $A$ with certain conditions on Engel sinks of elements of $C_G(a)$ for $a\in A^\#$. Here $q$ is a prime, $C_G(a)$ denotes the subgroup of $G$ formed by the fixed points of the automorphism $a$, while the symbol $A^{\#}$ stands for the set of nontrivial elements of $A$. Recall that an action of $A$ on $G$ is called coprime if $(|A|,|G|)=1$.  Our purpose is to establish the following theorems.

\begin{theorem}\label{q2finiteorder} Let $m$ be a positive integer, $q$ a prime, and $A$ an elementary abelian group of order $q^2$ acting coprimely  on a finite group $G$. If $\E(g)$ has at most $m$ elements for  every $g\in C_G(a)$ and every $a\in A^{\#}$, then the order of $\gamma_\infty(G)$ is bounded in terms of $m$ only.
\end{theorem}

\begin{theorem}\label{q2finiteposto} Let $m$ be a positive integer, $q$ a prime, and $A$ an elementary abelian group of order $q^2$ acting coprimely  on a finite group $G$. If $\langle\E(g)\rangle$ has rank at most $m$ for every $g\in C_{G}(a)$ and every $a\in A^{\#}$, then the rank of $\gamma_\infty(G)$ is bounded in terms of $m$ and $q$ only.
\end{theorem}

The surprising aspect of Theorem \ref{q2finiteorder} is that the order of $\gamma_\infty(G)$ turns out to be independent of the order of $A$. We do not know if a similar phenomenon also holds under the hypotheses of Theorem \ref{q2finiteposto}. Quite possibly, Theorem \ref{q2finiteposto} can be strengthened by showing that the rank of $\gamma_\infty(G)$ is bounded in terms of $m$ only. This seems an interesting question for future projects.

Formally, the proofs of Theorems \ref{q2finiteorder} and \ref{q2finiteposto} follow similar plans but in fact there are significant differences. In particular, the proof of Theorem \ref{q2finiteorder} uses the observation that if a finite group $G$ admits a coprime action by an elementary abelian $q$-group $A$ such that $C_G(a)$ has order at most $m$ for each $a\in A^{\#}$, then the order of $G$ is bounded in terms of $m$ only and is independent of the order of $A$ (see Lemma \ref{mbounded} in Section 3). We think that there is no analog of that observation for the case where $C_G(a)$ has rank at most $m$ for each $a\in A^{\#}$. This partially explains why in Theorem \ref{q2finiteposto} the rank of $\gamma_\infty(G)$ is bounded in terms of both $m$ and $q$.

Note that both parts of Theorem \ref{KStheorems} can be obtained as particular cases of Theorems \ref{q2finiteorder} and \ref{q2finiteposto}, respectively, where the action of $A$ on $G$ is trivial. On the other hand, Theorem \ref{KStheorems} is used in the proof of both Theorems \ref{q2finiteorder} and \ref{q2finiteposto}. 

In the next section we include some (mostly well-known) auxiliary lemmas. The proof of Theorem \ref{q2finiteorder} is given in Section 3. The proof of Theorem \ref{q2finiteposto} is given in Section 4.


   \section{Preliminaries}  
Throughout the paper we use, without special references, the  following well-known properties of coprime actions (see for example \cite[5.3.6,  6.2,2, 6.2.4]{DG} and \cite[Theorem 8.2.6]{KS-book}).  Let $A$ be a group  acting coprimely on a finite group $G$. We will use the usual notation for commutators $ [g,a]=g^{-1}g^a $ and $ [G,A]=\langle [g,a] | g\in G, a\in A\rangle$. Then:  
\vspace{5pt}
	
\noindent $ \bullet$   $[G,A]=[G,A,A] $; \\
\noindent $ \bullet$  If $N$ is  an $A$-invariant normal subgroup of $ G $, then  $C_{G/N}(A)=C_G(A)N/N$; \\
\noindent $ \bullet$   If  $A$ is a noncyclic abelian group, then $G$ is generated by the subgroups $C_G(B)$, where $A/B$ is cyclic; \\
\noindent $ \bullet$   $G$ has an $A$-invariant Sylow $p$-subgroup for each prime $p$ dividing $|G|$. If $G$ is soluble, then there exist $A$-invariant Hall $\pi$-subgroups of $G$, for any  nonempty  subset $\pi$ of prime divisors of $|G|$.

In what follows, let ${\bf r}(G)$ denote the rank of a finite group $G$, and we use the expression ``$\{a,b,\dots\}$-bounded'' to abbreviate the expression ``bounded from above by a function that depends on  $a,b,\dots$ only''.

\vspace{2pt}		 
We will  also require the following well-known facts.

\begin{lemma}\label{fact1}
	Let $N,H_1,\ldots,H_t$ be subgroups of a group $G$ with $N$ being normal. If $K=\langle H_1,\ldots, H_t\rangle$, then $[N,K]=[N,H_1]\cdots [N,H_t]$.
\end{lemma}

For the proof of the next lemma see, for example, \cite[Lemma 2.6]{KS-rank}.
\begin{lemma}\label{fact2}
	Suppose that a group $A$ acts by automorphisms on a group $G$. If $A=\langle a_1,\ldots,a_t\rangle$, then $[G,A]=[G,a_1]\cdots[G,a_t].$
\end{lemma}

The next  results relate Engel sinks in a finite group to coprime actions. They were established in \cite[Lemma 3.2]{KS-2018} and \cite[Lemma 2.7]{KS-rank}, respectively.
\begin{lemma}\label{[P,g]}
Let  $P$ be a finite $p$-subgroup of  a group $G$, and $g\in G$ a $p'$-element normalizing $P$. 
Then 
\begin{itemize}
\item[(1)] the order of $[P,g]$ is bounded in terms of the cardinality of the Engel sink $\E(g)$;
\item[(2)] the subgroup $[P, g]$ is contained in   $\langle \E(g)\rangle$.
\end{itemize}
\end{lemma}

In the next  lemma we collect two useful results. 
\begin{lemma}\label{|[U,V]|}
	Let  $p$ be a prime, $P $ a finite $p$-group, and $ A$ a $ p' $-group of automorphisms of $ P$.
\begin{itemize}
	\item[(1)] If $|[P,a]|\leq m$ for every $a\in A$, then $|A|$ and $|[P,A]|$ are $m$-bounded.
	\item[(2)] If {\bf r}$([P,a])\leq m$ for every $ a\in A $, then {\bf r}$ (A) $ and {\bf r}$ ([P,A]) $ are $ m $-bounded.
\end{itemize}
\end{lemma}

\begin{proof}
The proof of Part (2) can be found in \cite[Proposition 3.1]{KS-rank}. Let us prove Part (1).

 Without loss of generality  we can assume that $P=[P,A]$. It is sufficient to bound the order of $A$ since a bound on the order of $P$ will follow simply observing, by Lemma \ref{fact2}, that $[P,A]=\prod_{a\in A}[P,a]$. Thus, let us show that the order of $A$ is $m$-bounded.

 Let $B$ be the group of automorphisms that $A$ induces on the quotient $P/\Phi(P)$. By  \cite[Lemma 2.3]{KS-2016} the order of $B$ is $m$-bounded. Suppose that $B$ has order smaller than $A$. This means that $A$ contains a nontrivial  element $a$ which acts trivially on $P/\Phi(P)$. This in turn means that $P=C_P(a)\Phi(P)$. Since  the Frattini subgroup $\Phi(P)$ is a non-generating set, we conclude that $P=C_P(a)$ and so $a=1$. Thus, $|A|=|B|$ and we are done.
\end{proof}

The following result was obtained by Kov\'acs \cite{Kov} for soluble groups, and extended, to the general case, independently by Guralnick \cite{RGur} and Lucchini \cite{ALuc} using the classification of finite simple groups.
\begin{lemma}\label{Kovacs}
 If d is the maximum of the ranks of the Sylow subgroups of a finite group, then the rank of this group is at most d +1.
\end{lemma}

Let $ F(G) $ denote the Fitting subgroup of a group $G$. Write $ F_0(G)=1, F_1(G)=F(G) $ and let $ F_{i+1}(G) $ be the inverse image of $ F(G/F_i(G)) $. If $ G $ is soluble, then the least number $ h $ such that $ F_h(G)=G $ is called the Fitting height $ h(G) $ of $ G $.
The next result is well-known (see for example  \cite[Lemma 2.4]{AST} for the proof).

\begin{lemma}\label{gama_inf}
	If $G$ is a finite group of Fitting height 2, then $\gamma_\infty(G)=\prod_p[F_p,H_{p'}]$, where $F_p$ is a Sylow $p$-subgroup of $F(G)$ and $H_{p'}$ is a Hall $p'$-subgroup of $G$.
\end{lemma}

\section{Proof of Theorem \ref{q2finiteorder}}
In this section we  present a proof of Theorem \ref{q2finiteorder}. 
Let us start with the following lemma.

\begin{lemma}\label{mbounded}
	Let $ m $ be a positive integer and $q$ a prime.  Let $ A $ be an elementary abelian group of order $ q^2 $ acting coprimely on a finite group $ G $ in such a manner that $| C_G(a)|\leq m$ for every $ a\in A^{\#} $. Then the order of $ G $ is $ m $-bounded.
\end{lemma}
\begin{proof}
Note that $A$ normalizes some Sylow $p_i$-subgroup $P_i$ of $G$, for each prime divisor $p_i$ of $|G|$. Since $|G|=|P_1|\cdots|P_s|$, it is enough to show that  $|P_i|$ is $m$-bounded for every $p_i$. Indeed, let $C$ be the maximum possible value for $|P_i|$, which is an $m$-bounded number. Then every  $p_i\leq C$ and, so $|G|$ is $m$-bounded as well. 

Thus, without loss of generality, we can assume that $ G $ is a $ p $-group for some $ p\neq q. $ Then $ G=\prod_{a\in A^{\#}}C_G(a) $ and so $ G $ has order at most $ m^{q+1} $, since $A$ has exactly $q + 1$  cyclic subgroups.

 If  $ q\leq m! $, then  $ |G|\leq m^{m!+1}$, and we are done.  Assume   that $ q > m!$. Let $ a\in A^{\#} $. Note that  the centralizer $ C_G(a) $ is $ A $-invariant since $A$ is abelian.  Observe that the order of any automorphism of $ C_G(a) $ is at most $(m-1)!$ and therefore $A$ induces the trivial action on $ C_G(a) $. In other words $ C_G(a) \leq  C_G(A) $.  This implies that  all centralizers are equal and, in particular, that $ G=C_G(a) $ for any $ a\in A^{\#} $.  Hence $ |G|\leq m$, as desired. 
\end{proof}

The first step in the proof of Theorem \ref{q2finiteorder} deals with the case where $G$ is soluble. We  require the following result due to  Thompson \cite{T} (see also \cite[Corollary 3.2]{Turull}).
\begin{theorem}\label{Turull}
	Let $ A $ be a  soluble coprime group of automorphisms of a finite  soluble group $ G $. Then the Fitting height of $ G $ is bounded in terms of $ h(C_G(A)) $ and the number of prime factors of $ |A| $ counting multiplicities.
\end{theorem}

\begin{lemma}\label{casesoluorder}
	Theorem \ref{q2finiteorder} holds if $G$ is a  soluble group.
\end{lemma}
\begin{proof} 

Suppose that $G$ is soluble. It follows from  Theorem \ref{KStheorems}(1)  that  $\gamma_\infty(C_G(a))$ has  $m$-bounded order, for each $a$ in $A^{\#}$. Therefore $h(C_G(a))$ is $m$-bounded for any $a$ in $A^{\#}$ and, by  Theorem \ref{Turull},  $h(G)$ is $m$-bounded as well. The lemma now can be proved by induction on $h(G)$. 

If $h(G)=1$, then the result is obvious.  Suppose that $h(G)=2$.
By Lemma \ref{gama_inf},  we have  $\gamma_{\infty}(G)=\prod_p[F_p,H_{p'}] $, where $ F_p $ is a Sylow $ p $-subgroup of $ F(G) $, while $ H_{p'} $ is a Hall $ p' $-subgroup of $ G $ that can be chosen $A$-invariant  and the product is taken over all prime divisors of $ |F(G)| $.

First observe that  $ |[F_p,H_{p'}]|= |[F_p,\overline{H}_{p'}]|$, where  $\overline{H}_{p'}=H_{p'}/C_{H_{p'}}(F_p)$.
By Lemma  \ref{[P,g]}(1), the order of  $ [F_p,h] $ is $ m $-bounded, for any $h\in C_{\overline{H}_{p'}}(a)$. It follows from Lemma \ref{|[U,V]|}(1)  that 
the order of $[F_p,C_{\overline{H}_{p'}}(a)]$ is $ m $-bounded for any $ a\in A^{\#}$. 
Since $ C_{[F_p,\overline{H}_{p'}]}(a)\leq [F_p,C_{\overline{H}_{p'}}(a)]$ for any $ a\in A^{\#} $, 
in view of Lemma \ref{mbounded}, we  conclude that $ |[F_p,\overline{H}_{p'}]| $ is $ m $-bounded.   

Let $C$ be the maximum possible value for $|[F_p,H_{p'}]|$, which is an $m$-bounded number. Suppose that $p>C$. Since $F_p$ cannot have nontrivial subgroups of order at most $C$, deduce that $[F_p,H_{p'}]=1$. Let $p_1,\ldots,p_s$ be the primes dividing $|\gamma_\infty(G)|$. It follows from the observation above that each $p_i$ is at most $C$. Therefore $|\gamma_{\infty}(G)|\leq C^s$, which is an  $m$-bounded number, as claimed.

Finally, suppose that $h(G)>2$ and let $N = F_2(G)$ be the second term of the Fitting series of $G$. It is clear that the Fitting height of $G/\gamma_{\infty}(N)$ is $h(G)-1$ and $\gamma_{\infty}(N)\leq \gamma_{\infty}(G)$. The above argument shows that $|\gamma_{\infty}(N)|$ is $m$-bounded. Hence, by induction the order of $\gamma_{\infty}(G)/\gamma_{\infty}(N)$ is $m$-bounded. The result follows. 
\end{proof}

In order to deal with  the case  when  $G$ is a nonsoluble group, first  we consider the special case where $G$ is a direct   product of  nonabelian simple groups.   

A proof of the following theorem can be found, for example, in \cite[Corollary 3.5]{AS}.

\begin{theorem}\label{solub}
Let $q$ be a prime and $A$ an elementary abelian group of order $q^2$ acting on a finite $q'$-group $G$ in such a way that $C_G(a)$ is nilpotent for every nontrivial $a\in A$. Then $G$ is soluble.
\end{theorem}

In the next lemma we will also use the fact that if $A$ is any coprime group of automorphisms of a finite simple group, then $A$ is cyclic (see, for example, \cite[Lemma 2.7]{PS-Gu}).

Let $M_1=M_1(m)$ be the maximal possible value of $|\gamma_\infty(G)|$ in Theorem \ref{KStheorems}(1) and let $M_2=M_2(m)$ be the maximal possible value of $|\gamma_\infty(G)|$ in Lemma \ref{casesoluorder}. Set $M=\mathrm{max}\{M_1,M_2\}$.

\begin{lemma}\label{order-simpleCase}
	Let $G$  and $A$ be as in Theorem  \ref{q2finiteorder}.  Moreover assume  that $ G=S_1\times \cdots\times S_t$ is a direct product of $ t $ nonabelian simple groups $ S_i $. Then the order of  $ G $ is $m$-bounded.
\end{lemma}

\begin{proof}
First, we prove that each direct factor $S_i$ has order at most $M$. If $S_i$ is $A$-invariant, then $S_i$ is contained in $C_G(a) $ for some $a\in A^{\#}$ and so $S_i\leq\gamma_{\infty}(C_G(a))$.  Hence $|S_i|\leq M$, since  by Theorem \ref{KStheorems}(1)  and the above observation $|\gamma_{\infty}(C_G(a))|\leq M$. Suppose now that $S_i$ is not $A$-invariant. Choose $a\in A^{\#} $ such that $S_i^a\neq S_i$. Write $T=S_i\times S_i^a\times\cdots\times S_i^{a^{q-1}}$. Since $S_i^a\neq S_i$, it is easy to see that $C_T(a)$ is exactly the ``diagonal'' subgroup of $T$, that is, 
	$$C_T(a)=\{(g,g^a,\ldots,g^{a^{q-1}})\mid g\in S_i\}.$$ 
Thus, $C_T(a)$ is isomorphic to $S_i$ and Theorem \ref{KStheorems}(1) implies that $|S_i|\leq M$, as claimed.

Next, we show that if $q\geq M!$, then $A$ acts trivially on $G$.
Suppose that this is false and $q\geq M!$ while $A$ acts nontrivially on $G$.

Let $K$ be a minimal normal $A$-invariant subgroup of $G$. Thus, $K$ is direct product of subgroups $S^a$, where $S$ is one of $S_i$ and $a$ ranges through $A$. As shown above, if $K=S$, then $|K|\leq M$. Since for any $a\in A^{\#}$ the order of $a$ is bigger that $M!$,  it follows that $K\leq C_G(A)$, a contradiction. Suppose now  that $K\not=S$. 
In this case either $K$ is a product of $q$ simple factors or a product of $q^2$ ones. In the former case we have $K=S\times S^a\cdots\times S^{a^{q-1}}$ for some $a\in A$ and there exists $b\in A$ such that $S^b=S$. If $b$ centralizes $S$, then $K\leq C_G(b)$ and so $|K|\leq M$ whence again we deduce that $A$ centralizes $K$ since $q\geq M!$. This is a contradiction because $S\neq S^a$. Hence, $b$ does not centralize $S$. Note that $C_K(a)\cong S$ is the ``diagonal" subgroup of $K$. Observe that $C_K(a)$ is $A$-invariant and the order of $C_K(a)$ is at most $M$. Taking into account that $q\geq M!$ we deduce that $A$ acts trivially on $C_K(a)$. In particular, $C_K(a)\leq C_K(b)$. On the other hand, $$C_K(b)=C_S(b)\times C_S(b)^a\times\dots C_S(b)^{a^{q-1}}$$ and so $C_K(a)$ is not contained in $C_K(b)$.

Therefore we can assume that $K$ is a product of $q^2$ simple factors which are  transitively permuted by $A$. Recall that  minimal nonnilpotent finite groups are soluble \cite[Theorem 9.1.9]{Rob}. Therefore $S$ contains a soluble subgroup $D$ which is not nilpotent.  Thus $|\gamma_\infty(D)|\geq 3$. Then $D_0=\langle D^A\rangle$ is a soluble $A$-invariant subgroup of $K$ such that $$|\gamma_\infty(D_0)|\geq3^{q^2}.$$ This is a contradiction since $q\geq M$ while $|\gamma_\infty(D_0)|\leq M$. 

This proves that either $A$ acts trivially on $G$  or $q\leq M!$. In the former case $|G|\leq M$ and we are done. Thus, assume that $q\leq M!$.

Write $G=K_1\times\cdots\times K_s $ where each $K_i$ is a minimal normal $A$-invariant subgroup of $G$. Since, as  proved above,  each  factor $S_i$ has order at most $M$ and  $q$ is now $M$-bounded,  each $K_i$ is of $m$-bounded order. Therefore, it is sufficient to bound $s$.

By Theorem \ref{solub}, for every $i$ there exists $a\in A^{\#} $ such that $C_{K_i}(a)$ is not nilpotent. Therefore, $ \gamma_{\infty}(C_{K_{i}}(a))\neq1$. Since 	$\gamma_{\infty}(C_G(a))$ has order at most $M$, it follows that $\gamma_{\infty}(C_{K_{i}}(a))$ can be nontrivial for at most $M$ indexes $i$. Taking into account that there are only $q+1$ nontrivial proper subgroups in $A$, we conclude that $s\leq (q+1)M\leq (M!+1)M$.
\end{proof}

Every finite group $G$ has a normal series each of whose factors either is soluble or is a nonempty direct product of nonabelian simple groups. The nonsoluble length $\lambda(G)$ of a finite group $G$ is defined as the minimum number of nonsoluble factors in a normal series of that kind.

We will require the following two  results.

\begin{theorem}[{\cite[Theorem 1.3]{KS-nonsoluble}}]\label{KS-nonsoluble}
	Let $ A $ be a coprime group of automorphisms of a finite group  $ G $. Then the nonsoluble length of $ G $ is bounded in terms of $ \lambda(C_G(A)) $ and the number of prime factors of $ |A| $ counting multiplicities.
\end{theorem}

\begin{lemma}[{\cite[ Lemma 3.2]{PEA}}]\label{simple2generator}
	Let $ K $ be a nonabelian finite simple group and $ p $ a prime. There exists a prime $ s $, different from $ p $, such that $ K $ is generated by two Sylow $ s $-subgroups.
\end{lemma}

We are now ready to  prove Theorem \ref{q2finiteorder} in the general case.

\begin{proof}[Proof of Theorem \ref{q2finiteorder}] In view of  Lemma \ref{casesoluorder} we can assume that $ G $ is not soluble. By Theorem \ref{KStheorems}(1), the order of $ \gamma_{\infty}(C_G(a)) $ is $ m $-bounded for any $a\in A^{\#}$.  Therefore  $\lambda(\gamma_{\infty}(C_G(a)))$ is $ m $-bounded, and so $ \lambda(C_G(a)) $ is $ m $-bounded as well. Thus Theorem \ref{KS-nonsoluble} guarantees that $ \lambda(G) $ is $ m $-bounded, and we can use induction on $ \lambda(G) $.

First, assume that $ \lambda(G)=1 $ and $ G=G' $.  Let $R(G)$ be the  soluble radical of $ G $.  It follows that $ G/R(G) $ is a product of nonabelian simple groups.  Since $ R(G) $ is an  $ A $-invariant soluble group, Lemma \ref{casesoluorder} ensures that the order of $\gamma_{\infty}(R(G))$ is $m$-bounded. Note that if the nilpotent residual of $G/\gamma_{\infty}(R(G))$ has $m$-bounded order, then $ \gamma_{\infty}(G) $ has $m$-bounded order too. Thus, without loss of generality, we assume that $\gamma_{\infty}(R(G))=1$ and  $R(G)=F(G)$. 

Next we will show that the order of $[F(G),G]$ is $m$-bounded. Since $F(G)$ is nilpotent we have
$$[F(G),G]=[P_1,G]\cdots [P_l,G],$$ 
where the subgroups $P_i$ are Sylow $p_i$-subgroups of $F(G)$. It is sufficient to show that the order of $[P,G]$ is $m$-bounded  for  each $p$-Sylow subgroup $P$ of $F(G)$. Note that, considering the quotient of $G$ by  the Hall $p'$-subgroup of $F(G)$ and taking into account that the soluble radical of $G$ is nilpotent, we can assume that $F(G)=P$ is a $p$-group.

Since $\lambda(G)=1$ and $G=G'$, the quotient $G/F(G)$ is a direct product of nonabelian simple groups. By Lemma \ref{order-simpleCase}   $ G/P $  has $ m $-bounded order. Write $G=G_1G_2\ldots G_t$, where  $G_i/P$ is a minimal normal $ A $-invariant subgroup of $G/P$. Obviously $[P,G]=\prod_i [P,G_i]$. Thus it is sufficient to show that the order of $[P,G_i]$ is $m$-bounded. Hence, without loss of generality, we can assume that $G=G_1$, that is, $G/P$ is a product of isomorphic nonabelian simple groups transitively permuted by $A$.  By Lemma \ref{simple2generator} the quotient  $ G/P $ is generated by the image of two Sylow $s$-subgroup $H_1$ and $H_2$, where $s$ is a prime different from $p$. Furthermore, $H_1$ and $H_2$ are conjugate of an $A$-invariant Sylow $s$-subgroup $K$.

Note that, by Lemma \ref{casesoluorder}, the order of $ \gamma_{\infty}(PK) $ is $ m $-bounded because $ PK $ is $ A $-invariant and soluble. We also have $ [P,K]=[P,K,K] $ and consequently $ [P,K]\leq \gamma_{\infty}(PK) $. Since $ H_1 $ and $ H_2 $ are conjugate to $ K, $ it follows that both $ [P,H_1] $ and $ [P,H_2] $ have  $ m $-bounded order.

Let $ H=\left<H_1, H_2\right> $. Thus $ G=PH $. Since $ G=G' $ we  deduce that $ G=[P,H]H $ because $[P,H]H $ is normal and the quotient $G/[P,H]H$ is  a $p$-group. Furthermore, $ [P,G]=[P,H] $. Indeed, we have $$[P,G]=[P,[P,H]H]=[P,[P,H]][P,H].$$  Since $ [P,H]$ is normal in $P $ and $ G=PH $, the subgroup $ [P,H] $ is also normal in $ G $ and so $ [[P,H],P]\leq [P,H]. $

Observe that $[P,H]=[P,H_1][P,H_2]$. Thus the order of $[P,H]$ is $m$-bounded. Passing to the quotient $G/[P,G]$ we can assume that $P=Z(G)$. In view of  Lemma \ref{order-simpleCase}, we know   that  $ G/Z(G) $ has $ m $-bounded order. By quantitative version of Schur's Theorem (see, for example, \cite[p. 102]{Rob2})  the order of  $ G'=G $ is $ m $-bounded as well.

Now we deal with the case where $ \lambda(G)=1 $ and $ G\neq G' $. Let $ G^{(d)} $ be the last term of the derived series of $ G $. The argument in the previous paragraph shows that $ G^{(d)} $ has  $ m $-bounded order. Hence, the order of $ \gamma_{\infty}(G) $ is   $ m $-bounded since $ G/G^{(d)} $ is soluble. This proves the theorem in the particular case where $ \lambda(G)\leq 1. $

Assume that $ \lambda(G)\geq 2.  $ Let $ T $ be a characteristic subgroup of $ G $ such that $ \lambda(T)=\lambda(G)-1 $ and $ \lambda(G/T)=1. $ By induction, the order of $ \gamma_{\infty}(T) $ is $ m$-bounded. Since  $ \lambda(G/\gamma_{\infty}(T))=1$,  the order of the nilpotent residual of  $ G/\gamma_{\infty}(T) $ is  $ m $-bounded too. Consequently, we  conclude that $ \gamma_{\infty}(G) $ has   $ m$-bounded order. This completes  the proof.
\end{proof}

\section{Proof of Theorem \ref{q2finiteposto}}
In this section we  prove Theorem \ref{q2finiteposto}. Similarly to what was done in the previous section,  the first step in the proof  deals with the case of soluble groups.
\begin{lemma}\label{casesolurank}
    Theorem \ref{q2finiteposto} holds if $G$ is a  soluble group.
\end{lemma}
\begin{proof}
Assume that $G$ is soluble.
By Theorem \ref{KStheorems}(2) the rank of  $\gamma_{\infty}(C_G(a))$ is $m$-bounded for each $a\in A^{\#}$.  In \cite[Lemma 8]{KM} it was proved that a finite soluble group of rank $r$ has $r$-bounded Fitting height. Hence we deduce that $h(\gamma_{\infty}(C_G(a))) $ is $m$-bounded and so, in particular,   $C_G(a)$ has $m$-bounded Fitting height for any $a\in A^{\#}$. Now Theorem \ref{Turull} implies that $ h(G) $ is $m$-bounded. The lemma now can be proved by induction on $ h(G)$. If $h(G)=1$, then the result is obvious. 

 Next, we are going to establish that if $h(G)=2$, then the rank of $\gamma_\infty(G)$ is $ \{m,q\}$-bounded.  By Lemma \ref{gama_inf},  we have  $\gamma_{\infty}(G)=\prod_p[F_p,H_{p'}] $, where $ F_p $ is a Sylow $ p $-subgroup of $ F(G) $, while $ H_{p'} $ is a Hall $ p' $-subgroup of $ G$ that can be chosen $A$-invariant  and the product is taken over all prime divisors of $ |F(G)| $.       It follows from Lemma \ref{[P,g]}(2) that $[F_p, h]$ has $ m$-bounded rank for any $h\in C_{H_{p'}}(a)$. Hence, by Lemma \ref{|[U,V]|}(2), we deduce that  ${\bf r} ([F_p,C_{H_{p'}}(a)]) $ is $m$-bounded for each $a\in A^{\#}$.  Since $H_{p'}=\langle C_{H_{p'}}(a) \mid a\in A^{\#}\rangle$, Lemma \ref{fact1} tells us that $$[F_p,H_{p'}]=\prod_{a\in A^{\#}}[F_p,C_{H_{p'}}(a)].$$ Therefore the rank of  $ [F_p,H_{p'}] $ is $ \{m,q\} $-bounded.   By Lemma \ref{Kovacs} we conclude that $ \gamma_{\infty}(G) $ has $ \{m,q\} $-bounded rank, as well.

Finally, suppose that $h(G)>2$ and let $N = F_2(G)$. 
 By induction the rank of $\gamma_{\infty}(G)/\gamma_{\infty}(N)$ is $\{m,q\}$-bounded. The result follows since ${\bf r}(\gamma_{\infty}(G))\leq {\bf r}(\gamma_{\infty}(G)/\gamma_{\infty}(N))+ {\bf r}(\gamma_{\infty}(N))$.
\end{proof}

Now we  consider the special case where $G$ is a direct   product of  nonabelian simple groups. Most of the  argument follows the scheme as in the proof of Lemma \ref{order-simpleCase}.

\begin{lemma}\label{rk-simpleCase} Let $G$ and $A$ be as in Theorem \ref{q2finiteposto}. Assume also that $ G=S_1\times\cdots\times S_t$ is a direct product of $ t $ nonabelian simple groups $ S_i $. Then $ t $ is  $ \{m,q\} $-bounded and the rank of each $S_i$ is $m$-bounded.
\end{lemma}

\begin{proof} First, we prove that each direct factor $ S_i $ has $m$-bounded rank.  If $ S_i $ is $ A $-invariant, then $ S_i $ is contained in $ C_G(a) $ for some $ a\in A^{\#} $ and so $ S_i\leq \gamma_{\infty}(C_G(a)) $. Hence $ S_i $ has $m$-bounded rank by Theorem  \ref{KStheorems}(2). Suppose now that $ S_i $ is not $ A $-invariant. Choose $ a\in A^{\#} $ such that $ S_i^a\neq S_i. $ Write $ S=S_i\times S_i^a\times \cdots\times S_i^{a^{q-1}} $. Since $ S_i^a\neq S_i $ it is easy to  see that $ C_S(a) $ is exactly the ``diagonal'' subgroup of $ S $, that is, 
	$$C_S(a)=\{(g,g^a,\ldots,g^{a^{q-1}}) \mid g\in S_i\}.$$ 
	Thus, $C_S(a)$ is isomorphic to $ S_i $ and  Theorem \ref{KStheorems}(2) implies that $ S_i $ has $m$-bounded rank, as claimed.
		
	Now, we will show that $ t $ is an $ \{m,q\} $-bounded  number. Write $ G=K_1\times \cdots\times K_s $ where each $ K_i $ is a minimal normal $ A $-invariant subgroup of $ G $. Then each $ K_i $ is a product of at most $ |A| $ simple factors and so $ t\leq |A|s $. Therefore, it is sufficient to bound $ s $.

	Let $ S_j $ be a direct factor of $ K_i $. Let us show that  for every $ i $ there exists $ a\in A^{\#} $ such that $ C_{K_i}(a) $ contains a subgroup isomorphic to some $S_j$. If $ S_j $ is $ A $-invariant, then $ S_j $ is contained in $ C_{K_i}(a) $ for some $ a\in A^{\#} $. Suppose now that $ S_j $ is not $ A $-invariant. Choose $ b\in A^{\#} $ such that $ S_j^b\neq S_j. $ As shown above, the centralizer $ C_S(b) $ is exactly the ``diagonal'' subgroup of $ S=S_j\times S_j^b\times \cdots\times S_j^{b^{q-1}} $. Hence $ C_S(b) $  is isomorphic to $ S_j $ and it is contained in $ C_{K_i}(b) $, as claimed.
	
	Therefore, for every $i$, there exists $a\in A^{\#}$ such that  $ \gamma_{\infty}(C_{K_{i}}(a)) $ has even order, so in particular $ \gamma_{\infty}(C_{K_{i}}(a))\neq 1 $. Note that  ${\bf r}(\gamma_{\infty}(C_G(a))) $ is $m$-bounded by Theorem  \ref{KStheorems}(2), for any $a\in A^{\#}$. Let $r$ be the maximal value of ${\bf r}(\gamma_{\infty}(C_G(a))) $  when $a$ ranges through $A$, that is an $\{m,q\}$-bounded number.  Since  $\gamma_{\infty}(C_G(a))= \gamma_{\infty}(C_{K_1}(a))\times \cdots \times \gamma_{\infty}(C_{K_{s}}(a))$,    it follows that $s\leq r$, so it is $m$-bounded as desired.
\end{proof} 

We are now ready to  prove Theorem \ref{q2finiteposto} in the general case.
\begin{proof}[Proof of Theorem \ref{q2finiteposto}] In view of  Lemma \ref{casesolurank} we can assume that $ G $ is not soluble. By Theorem \ref{KStheorems}(2), the rank of $\gamma_{\infty}(C_G(a))$ is $m$-bounded for any  nontrivial element $a$ in $A$. 
Now \cite[Lemma 8]{KM} tells us that each soluble subgroup of $ \gamma_{\infty}(C_G(a)) $ has $ m $-bounded Fitting height. It was established in \cite[Corollary 1.2]{KS-non-p-soluble}  that the nonsoluble length  of a finite group does not exceed the maximum Fitting height of its soluble subgroups.  Therefore $\lambda(\gamma_{\infty}(C_G(a)))$ is $m$-bounded, and so $\lambda(C_G(a))$ is $m$-bounded as well. Thus Theorem \ref{KS-nonsoluble} guarantees that $\lambda(G)$ is $m$-bounded, and we can use induction on $\lambda(G)$.

First, assume that $\lambda(G)=1$ and $G=G'$. Let $R(G)$ be the soluble radical of $G$. It follows that $G/R(G)$ is a product of nonabelian simple groups.  Since $R(G)$ is an $A$-invariant soluble group, Lemma \ref{casesolurank} ensures that {\bf r}$(\gamma_{\infty}(R(G)))$ is $\{m,q\}$-bounded. Note that if the nilpotent residual of $G/\gamma_{\infty}(R(G))$ has $\{m,q\}$-bounded rank, then $ \gamma_{\infty}(G) $ has $\{m,q\}$-bounded rank too. Thus, without loss of generality, we assume that $\gamma_{\infty}(R(G))=1$ and  $R(G)=F(G)$. 

Next we will show that the rank of $[F(G),G]$ is $\{m,q\}$-bounded. Since $F(G)$ is nilpotent we have $$[F(G),G]=[P_1,G]\cdots [P_l,G],$$ 
where the subgroups $P_i$ are Sylow $p_i$-subgroups of $F(G)$. In view of Lemma \ref{Kovacs} it is enough to show that the rank of $[P,G]$ is $\{m,q\}$-bounded  for each Sylow $p$-subgroup $P$ of $F(G)$. Note that, considering the quotient of $G$ by the Hall $p'$-subgroup of $F(G)$ and taking into account that the soluble radical of $G$ is nilpotent, we can assume that $F(G)=P$ is a $p$-group.

Since $\lambda(G)=1$ and $G=G'$, the quotient $G/P$ is a direct product of nonabelian simple groups. By Lemma \ref{rk-simpleCase} $ G/P $  has $ \{m,q\} $-boundedly many  simple factors. Write $G=G_1G_2\ldots G_t$, where $t$ is  $\{m,q\}$-bounded and $G_i/P$ is a minimal normal $ A $-invariant subgroup of $G/P$. Obviously $[P,G]=\prod_i [P,G_i]$. Thus it is sufficient to show that the rank of $[P,G_i]$ is $\{m,q\}$-bounded. Hence, without loss of generality, we can assume that $G=G_1$, that is, $G/P$ is a product of isomorphic nonabelian simple groups transitively permuted by $A$. By Lemma \ref{simple2generator} the quotient  $ G/P $ is generated by the image of two Sylow $ s $-subgroup $ H_1 $ and $ H_2 $, where $ s $ is a prime different from $ p $. Furthermore,  $ H_1 $ and $ H_2 $ are conjugate of an $ A $-invariant Sylow $ s $-subgroup $K$.
	
Observe that, by Lemma \ref{casesolurank}, {\bf r}$ (\gamma_{\infty}(PK)) $ is $ \{m,q\} $-bounded because $ PK $ is $ A $-invariant and soluble. We also have $ [P,K]=[P,K,K] $ and consequently $ [P,K]\leq \gamma_{\infty}(PK) $. Since $ H_1 $ and $ H_2 $ are conjugate to $ K, $ it follows that both $ [P,H_1] $ and $ [P,H_2] $ have  $ \{m,q\} $-bounded rank.
	
Let $ H=\left<H_1, H_2\right> $. Thus $ G=PH $. Since $ G=G' $ we can deduce that $ G=[P,H]H $ because $[P,H]H $ is normal and the quotient $G/[P,H]H$ is  a $p$-group. Furthermore, $ [P,G]=[P,H] $. Indeed, we have $$[P,G]=[P,[P,H]H]=[P,[P,H]][P,H].$$  Since $ [P,H]$ is normal in $P $ and $ G=PH $, the subgroup $ [P,H] $ is also normal in $ G $ and so $ [[P,H],P]\leq [P,H]. $
	
Note that $ [P,H]=[P,H_1][P,H_2] $. Thus the rank of $ [P,H] $ is $ \{m,q\} $-bounded. Passing to the quotient $ G/[P,G] $ we can assume that $ P=Z(G) $. In view of Lemma \ref{rk-simpleCase}, we  know that $ G/Z(G) $ has  $ \{m,q\} $-bounded rank. By a theorem of Lubotzky and Mann \cite{luma} (see also \cite[Theorem 1.1]{P-Kur})  the rank of $ G' $ is $ \{m,q\} $-bounded as well. Since $ G=G' $ we conclude that the rank of $ G $ is $ \{m,q\} $-bounded. 
	
Now we deal with the case where $ \lambda(G)=1 $ and $ G\neq G' $. Let $ G^{(d)} $ be the last term of the derived series of $ G $. The argument in the previous paragraph shows that $ G^{(d)} $ has  $ \{m,q\} $-bounded rank. Hence, in virtue of  Lemma \ref{casesolurank}, the rank of $ \gamma_{\infty}(G) $ is  $ \{m,q\} $-bounded since $ G/G^{(d)} $ is soluble. This proves the theorem in the particular case where $ \lambda(G)\leq 1. $
	
Assume that $ \lambda(G)\geq 2.  $ Let $ T $ be a characteristic subgroup of $ G $ such that $ \lambda(T)=\lambda(G)-1 $ and $ \lambda(G/T)=1. $ By induction, the rank of $ \gamma_{\infty}(T) $ is $ \{m,q\} $-bounded. Since  $ \lambda(G/\gamma_{\infty}(T))=1$,  the rank of the nilpotent residual of  $ G/\gamma_{\infty}(T) $ is  $ \{m,q\} $-bounded too. Consequently, we  conclude that $ \gamma_{\infty}(G) $ has   $ \{m,q\} $-bounded rank. This completes  the proof.
\end{proof}


\begin{thebibliography}{99}


\bibitem{AS} C. Acciarri, P. Shumyatsky,  Double automorphisms of graded Lie algebras. {\it J. Algebra}, {\bf 387} (2013), 1--10.
\vspace{0.1cm}


\bibitem{AST} C. Acciarri, P.  Shumyatsky,  A. Thillaisundaram, Conciseness of coprime commutators in finite groups, {\it Bull. Aust. Math. Soc.}{\bf 89} (2014), 252--258.
\vspace{0.1cm}


\bibitem{DG}
D. Gorenstein, \ {\it Finite Group}, Chelsea Publishing Company, New York, 1980.


\vspace{0.1cm}

\bibitem{RGur}
R. Guralnick, On the number of generators of a finite group, {\it Arch. Math.} {\bf 53} (1989), 521--523.

\vspace{0.1cm}

\bibitem{PS-Gu} R. Guralnick, P. Shumyatsky,\ Derived subgroups of fixed points, {\it Israel J. Math.}  \textbf{126}, (2001), 345--362.


\vspace{0.1cm}


\bibitem{Kov} L.\,G. Kov\'acs, On finite soluble groups, {\it Math. Z.} {\bf103} (1968), 37--39.

\vspace{0.1cm}

\bibitem{KM} 
E.\,I. Khukhro, V.\,D. Mazurov, Finite groups with an automorphism of prime order whose centralizer has small rank,  {\it J. Algebra}, \textbf {301}  (2006),  474--492.

\vspace{0.1cm}

 
\bibitem{KS-nonsoluble} 
 E.\,I. Khukhro, P. Shumyatsky, On the length of finite groups and of fixed  points, {\it Proc. Amer. Math. Soc.} \textbf{ 143} (2015), 3781--3790.

\vspace{0.1cm}

\bibitem{KS-non-p-soluble} 
E.\,I. Khukhro, P. Shumyatsky, Nonsoluble and non-p-soluble length of finite  groups, {\it Israel J. Math.}  {\bf 207} (2015), 507--525.
 
\vspace{0.1cm}

\bibitem{KS-2016}
 E.\,I. Khukhro, P.  Shumyatsky,  Almost Engel finite and profinite groups,  {\it Internat. J. Algebra Comput.} \textbf {26}   (2016),  973--983.

\vspace{0.1cm}

\bibitem{KS-2018}
 E.\,I. Khukhro, P.  Shumyatsky,  Almost Engel compact groups,  {\it J. Algebra}, \textbf {500}  (2018),  439--456. 

\vspace{0.1cm}

\bibitem{KS-rank}
E.\,I. Khukhro, P.  Shumyatsky, Finite groups with Engel sinks of bounded rank,  {\it  Glasgow Math. J.}  (2018),  1--7. doi.org/10.1017/S0017089517000404. 

\vspace{0.1cm}

\bibitem{P-Kur}
L.\,A. Kurdachenko, P.  Shumyatsky, The ranks of central factor and commutator groups. {\it  Math. Proc. Camb. Phil. Soc.} {\bf 154} (2013), 63--69.

\vspace{0.1cm}


\bibitem{KS-book} 
H.\, Kurzweil, B.\, Stellmacher, \ {\it The Theory of
Finite Groups: An Introduction},  Springer-Verlag, New York, Inc.  2004.

\vspace{0.1cm}


\bibitem{luma}
A. Lubotzky, A. Mann,  Powerful $p$-groups I,  {\it J. Algebra}, \textbf {105}  (1987),  484--505. 

\vspace{0.1cm}

\bibitem{ALuc}
A. Lucchini, A bound on the number of generators of a finite group, {\it Arch. Math.}  {\bf 53} (1989), 313--317.

\vspace{0.1cm}
\bibitem{PEA}
E. Melo, A.\,S. Lima, P.  Shumyatsky,  Nilpotent residual of fixed points,  {\it Arch. Math.} (2018), 1--9.  doi.org/10.1007/s00013-018-1173-1.



\vspace{0.1cm}

\bibitem{Rob2}
D.\,J.\,S. Robinson, \textit{Finiteness Conditions and Generalized Soluble Groups}, Part 1, Springer, Berlin etc., 1972.
\vspace{0.1cm}


\bibitem{Rob}D.\,J.\,S. Robinson, \textit{A Course in the Theory of Groups}, 2nd ed. Springer-Verlag, New York, 1996.


\vspace{0.1cm}
\bibitem{T}J.G. Thompson, Automorphisms of solvable groups, {\it J. Algebra} \textbf{1} (1964), 259--267. 
\vspace{0.1cm}

\bibitem{Turull}
A. Turull, Fitting height of groups and of fixed points,  {\it J. Algebra} \textbf {86}  (1984),  555--566.


\vspace{0.1cm}


\bibitem{Zorn}M. Zorn, Nilpotency of finite groups, {\it Bull. Amer. Math. Soc.} {\bf 42} (1936), 485--486.
%

\end{thebibliography}
 \end{document}